\documentclass[a4paper, 10pt, notitlepage]{article}

\usepackage{amsthm, amsmath, amssymb, latexsym}
\usepackage{mathrsfs}

\theoremstyle{plain}
\newtheorem{theorem}{Theorem}
\newtheorem{lemma}{Lemma}
\newtheorem{corollary}{Corollary}

\theoremstyle{definition}
\newtheorem{definition}{Definition}
\newtheorem{example}{Example}

\theoremstyle{remark}
\newtheorem{remark}{Remark}

\DeclareMathOperator{\dom}{dom}
\DeclareMathOperator{\sign}{sign}
\DeclareMathOperator{\co}{co}

\author{M.V. Dolgopolik}
\title{Smooth exact penalty functions II: a reduction to standard exact penalty functions}

\begin{document}

\maketitle

\begin{abstract}
A new class of smooth exact penalty functions was recently introduced by Huyer and Neumaier. In this paper, we prove
that the new smooth penalty function for a constrained optimization problem is exact if and only if the standard
nonsmooth penalty function for this problem is exact. We also provide some estimates of the exact penalty parameter of
the smooth penalty function, and, in particular, show that it asymptotically behaves as the square of the exact
penalty parameter of the standard $\ell_1$ penalty function. We briefly discuss a simple way to reduce the exact
penalty parameter of the smooth penalty function, and study the effect of nonlinear terms on the exactness of this
function.
\end{abstract}

\section{Introduction}

The method of exact penalty functions \cite{HanMangasarian, DiPilloGrippo, BurkeExPen, Zaslavski} is a very appealing
technique for solving various constrained optimization problems, since it allows one to replace a constrained problem by
a single unconstrained optimization problem having the same optimal solutions. However, the equivalent unconstrained
problem is usually nonsmooth (even if the original problem is smooth), which makes the method of exact penalty functions
less attractive, especially for practitioners who are often not familiar with efficient methods for solving complicated
nonsmooth optimization problems.

Huyer and Neumaier \cite{HuyerNeumaier} proposed a new approach to exact penalization that allows one to overcome
nonsmoothness of exact penalty functions. Later on, this approach was modified \cite{Bingzhuang, WangMaZhou}, and
successfully applied to various constrained optimization and optimal control problems \cite{MaLiYiu, LinWuYu, LiYu,
JianLin, LinLoxton}. A new general approach to the construction and analysis of smooth exact penalty functions was
proposed in \cite{Dolgopolik}.

Let us recall the definition of the exact penalty function from \cite{WangMaZhou}. Consider the following constrained
optimization problem
\begin{equation} \label{NonlinearProgr}
  \min f(x) \quad \text{subject to} \quad F(x) = 0, \quad x \in [\underline{x}, \overline{x}],
\end{equation}
where $f \colon \mathbb{R}^n \to \mathbb{R}$ and $F \colon \mathbb{R}^n \to \mathbb{R}^m$ are smooth functions,
$\underline{x}, \overline{x} \in \mathbb{R}^n$ are given vectors, and
$$
  [\underline{x}, \overline{x}] = \big\{ x = (x_1, \ldots, x_n) \in \mathbb{R}^n \mid 
  \underline{x}_i \le x_i \le \overline{x}_i \quad \forall i \in \{ 1, \ldots, n \} \big\}.
$$
Given $a \in (0, + \infty]$, let $\phi \colon [0, a) \to [0, + \infty)$ be a convex continuously differentiable function
such that $\phi(0) = 0$ and $\phi'(t) > 0$ for all $t \in [0, a)$. Let also $w \in \mathbb{R}^m$ be arbitrary, and
$\beta \colon [0, \overline{\varepsilon}] \to [0, + \infty)$ with $\overline{\varepsilon} > 0$ be a continuously
differentiable nondecreasing function such that $\beta(t) = 0$ iff $t = 0$. Then one defines the new ``smooth'' penalty
function for the problem (\ref{NonlinearProgr}) as follows
\begin{equation} \label{Def_SmoothExactPenFunc}
  F_{\lambda}(x, \varepsilon) = \begin{cases}
    f(x), & \text{ if } \varepsilon = \Delta (x, \varepsilon) = 0, \\
    f(x) + \frac{1}{2 \varepsilon} \phi(\Delta(x, \varepsilon)) + \lambda \beta(\varepsilon), & \text{ if }
    \varepsilon > 0, \: \Delta(x, \varepsilon) < a, \\
    + \infty, &  \text{otherwise}.
  \end{cases}
\end{equation}
where $\lambda \ge 0$ is the penalty parameter, $\Delta(x, \varepsilon) = \| F(x) - \varepsilon w \|^2$ is the
constraint violation measure.  Finally, one replaces the problem (\ref{NonlinearProgr}) with the penalized problem 
\begin{equation} \label{FirstPenProblem}
  \min_{x, \varepsilon} F_{\lambda}(x, \varepsilon) \quad \text{subject to} \quad 
  (x, \varepsilon) \in [\underline{x}, \overline{x}] \times [0, \overline{\varepsilon}].
\end{equation}
Observe that the penalty function $F_{\lambda}(x, \varepsilon)$ depends on the additional parameter 
$\varepsilon \ge 0$, and $F_{\lambda}(x, \varepsilon)$ is smooth for any $\varepsilon \in (0, \overline{\varepsilon)}$
and $x$ such that $0 < \Delta(x, \varepsilon) < a$. Therefore one can apply standard algorithms of smooth optimization
to the penalty function (\ref{Def_SmoothExactPenFunc}) in order to find a globally/locally optimal solution of
penalized problem (\ref{FirstPenProblem}), which under natural assumptions (namely, constraint qualification) has the
form $(x^*, 0)$, where $x^*$ is a globally/locally optimal solution of problem (\ref{NonlinearProgr}). However, it
should be noted that the standard proofs of the exactness of the smooth penalty function 
(\ref{Def_SmoothExactPenFunc}) (i.e. proofs of the fact that all local and global minimizers of
(\ref{FirstPenProblem}) have the form $(x^*, 0)$) are rather complicated, and overburdened by technical details. A new
simple proof of the exactness of the penalty function (\ref{Def_SmoothExactPenFunc}) was given in \cite{Dolgopolik}.

The aim of this article is to continue the work started in \cite{Dolgopolik}, and present new simple methods for
studying the exactness of the penalty function of the form (\ref{Def_SmoothExactPenFunc}). Namely, we prove that
this penalty function is exact if an only if the standard nonsmooth $\ell_1$ penalty function for problem
(\ref{NonlinearProgr}) is exact, and provide some estimates of the exact penalty parameter of the penalty
function \eqref{Def_SmoothExactPenFunc} via the exact penalty parameter of the nonsmooth penalty function. In
particular, we demonstrate that the exact penalty parameter of the penalty function (\ref{Def_SmoothExactPenFunc}) with
$\varphi(t) \equiv \beta(t) \equiv t$ asymptotically behaves like the square of the exact penalty parameter of the
$\ell_1$ penalty function. We also discuss how to make the exact penalty parameter of the penalty
function (\ref{Def_SmoothExactPenFunc}) significantly smaller, and study how the nonlinear functions $\phi$ and $\beta$
affect the exactness of this penalty function. 

The paper is organised as follows. In Section~\ref{LinearCase}, we study the case $\phi(t) \equiv \beta(t) \equiv t$ in
detail. We prove that in this case the smooth penalty function is (locally or globally) exact if and only if the
corresponding nonsmooth penalty function is (locally or globally) exact, and provide several estimates of the exact
penalty parameter. In Section~\ref{NonlinearCase}, we study the effect of the nonlinear functions $\phi$ and $\beta$ on
the exactness of the penalty function \eqref{Def_SmoothExactPenFunc}.

\section{A reduction to standard exact penalty functions}
\label{LinearCase}

Let $X$ be a topological space, $(Y, d)$ be a metric space $f \colon X \to \mathbb{R} \cup \{ + \infty \}$ be a given
function, $\Phi \colon X \rightrightarrows Y$ be a set-valued mapping with closed values, and $A \subset X$ be a
nonempty set. Hereafter, we study the following optimization problem:
$$
  \min f(x) \quad \text{subject to} \quad y_0 \in \Phi(x), \quad x \in A,	\eqno{(\mathcal{P})}
$$
where $y_0 \in Y$ is a fixed element. Denote by $\Omega = \Phi^{-1}(y_0) \cap A$ the set of feasible points of the
problem ($\mathcal{P}$). Denote also $\dom f = \{ x \in X \mid f(x) < +\infty \}$. We suppose that 
$\Omega \cap \dom f \ne \emptyset$, and the function $f$ is bounded below on $\Omega$.

Introduce the ``smooth'' penalty function for the problem ($\mathcal{P}$):
$$
  F_{\lambda}(x, \varepsilon) = 
  f(x) + \varepsilon^{-1} d(y_0, \Phi(x))^2 + \lambda \varepsilon \quad \forall \varepsilon > 0,
$$
where $d(y_0, \Phi(x)) = \inf_{z \in \Phi(x)} d(y_0, z)$. From this point onwards, for any penalty function
$F_{\lambda}$ we suppose that $F_{\lambda}(x, 0) = f(x)$, if $x$ is feasible, and $F_{\lambda}(x, 0) = + \infty$
otherwise.

Alongside the problem ($\mathcal{P}$) we study the following extended penalized problem
\begin{equation} \label{PenalizedProblem}
  \min_{x, \varepsilon} F_{\lambda}(x, \varepsilon) \quad \text{subject to} \quad x \in A, \quad \varepsilon \ge 0.
\end{equation}
Note that only the constraint $y_0 \in \Phi(x)$ is included into the penalty function $F_{\lambda}$, while the
constraint $x \in A$ is taken into account explicitly.

Let us recall the concept of exactness of a penalty function that connects the initial problem ($\mathcal{P}$) with the
penalized problem (\ref{PenalizedProblem}). Denote $\mathbb{R}_+ = [0, + \infty)$.

\begin{definition}
Let $x^* \in \dom f$ be a point of local minimum of the problem ($\mathcal{P}$). The penalty function $F_{\lambda}$ is
said to be (locally) \textit{exact} at the point $x^*$ (or, to be more precise, at the point $(x^*, 0)$), if there
exists $\lambda \ge 0$ such that $(x^*, 0)$ is a point of local minimum of $F_{\lambda}$ on the set $A \times
\mathbb{R}_+$. The greatest lower bound of all such $\lambda$ is denoted by $\lambda^*(x^*)$ and is referred to as
\textit{the exact penalty parameter} (ex.p.p.) of the penalty function $F_{\lambda}$ at $x^*$.
\end{definition}

\begin{definition}
The penalty function $F_{\lambda}$ is called (globally) \textit{exact}, if there exists $\lambda \ge 0$ such that the
penalty function $F_{\lambda}$ attains a global minimum on the set $A \times \mathbb{R}_+$, and if 
$(x^*, \varepsilon^*)$ is a point of global minimum of $F_{\lambda}$ on $A \times \mathbb{R}_+$, then 
$\varepsilon^* = 0$. The greatest lower bound of all such $\lambda \ge 0$ is denoted by $\lambda^*$ and is referred to
as \textit{the exact penalty parameter} of the penalty fucntion $F_{\lambda}$.
\end{definition}

Thus, if the penalty function $F_{\lambda}$ is globally exact, then the problem ($\mathcal{P}$) and the penalized
problem (\ref{PenalizedProblem}) have the same globally optimal solutions for any sufficiently large $\lambda \ge 0$. To
be more precise, if $F_{\lambda}$ is globally exact and $\lambda > \lambda^*$, then $(x^*, \varepsilon^*)$ is a globally
optimal solution of the problem (\ref{PenalizedProblem}) iff $\varepsilon = 0$ and $x^*$ is a globally optimal solution
of the problem ($\mathcal{P}$). In other words, the global exactness of a penalty function means that the penalization
does not distort any information about globally optimal solutions of the original problem.

\begin{remark} \label{Rmrk_ExactPenCond}
One can show that the penalty function $F_{\lambda}$ is exact iff there exists $\lambda \ge 0$ such that
\begin{equation} \label{ExactPenCond}
  \inf_{x \in A, \varepsilon \ge 0} F_{\lambda}(x, \varepsilon) = \inf_{x \in \Omega} f(x),
\end{equation}
and $f$ attains a global minimum on $\Omega$. Furthermore, the greatest lower bound of all $\lambda \ge 0$ satisfying
(\ref{ExactPenCond}) is equal to the ex.p.p. $\lambda^*$.
\end{remark}

The two following theorems demonstrate that the penalty function $F_{\lambda}(x, \varepsilon)$ is exact if and only if
the standard penalty function $G_{\sigma}(x) = f(x) + \sigma d(y_0, \Phi(x))$ is exact. Thus, these theorems allow one
to apply a wide variety of methods of the theory of nonsmooth exact penalty functions \cite{HanMangasarian,
DiPilloGrippo,BurkeExPen, Zaslavski, Burke, DemyanovDiPilloFacchinei, Demyanov} to the study of the penalty function
$F_{\lambda}(x, \varepsilon)$.

\begin{theorem} \label{Thrm_LocalExactEquiv_Simple}
Let $x^* \in \dom f$ be a point of local minimum of the problem ($\mathcal{P}$), and the mapping $d(y_0, \Phi(\cdot))$
be continuous at $x^*$. Then the penalty function $F_{\lambda}(x, \varepsilon)$ is exact at $x^*$ if and only if 
the penalty function $G_{\sigma}(x) = f(x) + \sigma d(y_0, \Phi(x))$ is exact at this point, and
$$
  \lambda^*(x^*) = \frac{\sigma^*(x^*)^2}{4},
$$
where $\sigma^*(x^*)$ is the ex.p.p. of $G_{\sigma}$ at $x^*$.
\end{theorem}

\begin{proof}
Suppose that the penalty function $F_{\lambda}(x, \varepsilon)$ is exact at $x^*$, and fix an arbitrary 
$\lambda > \lambda^*(x^*)$. Then there exist a neighbourhood $U$ of $x$ and $\varepsilon_0 > 0$ such that
\begin{equation} \label{LocalExactnessDef}
  F_{\lambda}(x, \varepsilon) \ge F_{\lambda}(x^*, 0) \quad 
  \forall x \in U \cap A \quad \forall \varepsilon \in (0, \varepsilon_0].
\end{equation}
Since $x^*$ is a point of local minimum of the problem ($\mathcal{P}$), $x^*$ is feasible, i.e. 
$d(y_0, \Phi(x^*)) = 0$. Therefore applying the continuity of the function $d(y_0, \Phi(\cdot))$ at $x^*$ one gets that
there exists a neighbourhood $V$ of $x^*$ such that $V \subset U$ and
\begin{equation} \label{ContinSetValuedMap}
  d(y_0, \Phi(x)) < \sqrt{\lambda} \varepsilon_0 \quad \forall x \in V.
\end{equation}
From (\ref{LocalExactnessDef}) it follows that
\begin{equation} \label{MinLocalExactnessDef}
  \inf_{\varepsilon \in (0, \varepsilon_0)} F_{\lambda}(x, \varepsilon) \ge F_{\lambda}(x^*, 0) = f(x^*)
  \quad \forall x \in V \cap A.
\end{equation}
Let us compute the infimum on the left-hand side. If $x \in V \cap A$ is such that $y_0 \in \Phi(x)$, then
$F_{\lambda}(x, \varepsilon) = f(x) + \lambda \varepsilon$ for any $\varepsilon \ge 0$, and the infimum is equal to
$f(x)$. On the other hand, if $y_0 \notin \Phi(x)$, then
$$
  F_{\lambda}(x, \varepsilon) = f(x) + \frac{1}{\varepsilon} d(y_0, \Phi(x))^2 + \lambda \varepsilon \quad \forall
  \varepsilon > 0.
$$
Differentiating with respect to $\varepsilon$ one gets
$$
  \frac{d}{d \varepsilon} F_{\lambda}(x, \varepsilon) = - \frac{1}{\varepsilon^2} d(y_0, \Phi(x))^2 + \lambda
$$
Hence $F_{\lambda}(x, \varepsilon)$ decreases on $(0, \overline{\varepsilon})$ and increases on 
$(\overline{\varepsilon}, +\infty)$, where
$$
  \overline{\varepsilon} = \frac{d(y_0, \Phi(x))}{\sqrt{\lambda}}.
$$
Thus, $\overline{\varepsilon}$ is a point of global minimum of $F_{\lambda}(x, \cdot)$ on $(0, + \infty)$. Observe that
due to the choice of the neighbourhood $V$ (see (\ref{ContinSetValuedMap})) one has 
$\overline{\varepsilon} < \varepsilon_0$. Therefore
$$
  \min_{\varepsilon \in (0, \varepsilon_0)} F_{\lambda}(x, \varepsilon) = F_{\lambda}(x, \overline{\varepsilon}) =
  f(x) + 2 \sqrt{\lambda} d(y_0, \Phi(x)).
$$
Consequently, taking into account (\ref{MinLocalExactnessDef}) one obtains that
$$
  G_{2 \sqrt{\lambda}}(x) = \min_{\varepsilon \in (0, \varepsilon_0)} F_{\lambda}(x, \varepsilon) \ge
  f(x^*) = G_{2 \sqrt{\lambda}}(x^*) \quad \forall x \in V \cap A.
$$
Thus, the penalty function $G_{\lambda}$ is exact at $x^*$, and $2 \sqrt{\lambda} \ge \sigma^*(x^*)$, which due to the
arbitrary choice of $\lambda > \lambda^*(x^*)$ implies $2 \sqrt{\lambda^*(x^*)} \ge \sigma^*(x^*)$.

Suppose, now, that the penalty function $G_{\sigma}(x)$ is exact at $x^*$, and choose an arbitrary 
$\lambda > \sigma^*(x^*)^2 / 4$. Then there exists a neighbourhood $U$ of $x^*$ such that
$$
  G_{2 \sqrt{\lambda}}(x) \ge G_{2 \sqrt{\lambda}}(x^*) = f(x^*) \quad \forall x \in U \cap A.
$$
From the first part of the proof it follows that
$$
  \min_{\varepsilon > 0} F_{\lambda}(x, \varepsilon) = G_{2 \sqrt{\lambda}}(x).
$$
Therefore for any $x \in U \cap A$ and $\varepsilon \ge 0$ one has
$$
  F_{\lambda}(x, \varepsilon) \ge G_{2 \sqrt{\lambda}}(x) \ge f(x^*) = F_{\lambda}(x^*, 0).
$$
Thus, $F_{\lambda}(x, \varepsilon)$ is exact at $x^*$ and $\lambda^*(x^*) \le \sigma^*(x^*)^2 / 4$ by virtue of
the fact that $\lambda > \sigma^*(x^*)^2 / 4$ was chosen arbitrarily.	 
\end{proof}

Arguing in a similar way one can easily prove a global version of the theorem above.

\begin{theorem} \label{Thrm_GlobalExactEquiv_Simple}
The penalty function $F_{\lambda}(x, \varepsilon)$ is globally exact if and only if the penalty function 
$G_{\sigma}(x) = f(x) + \sigma d(y_0, \Phi(x))$ is globally exact, and
$$
  \lambda^* = \frac{(\sigma^*)^2}{4},
$$
where $\sigma^*$ is the ex.p.p. of $G_{\sigma}$.
\end{theorem}

Let us consider a simple particular case of the set-valued mapping $\Phi$. Namely, 
let $Y = \mathbb{R}^{m + l}$, $y_0 = 0$, and the set-valued mapping $\Phi$ have the form
\begin{equation} \label{SetValuedMap_MathProg}
  \Phi(x) = (h_1(x), \ldots, h_m(x)) \times \prod_{k = 1}^l [g_k(x), +\infty),
\end{equation}
where $h_i, g_k \colon X \to \mathbb{R}$ are given function, and $\prod$ stands for the Cartesian product. Thus, 
the inclusion $y_0 \in \Phi(x)$ is equivalent to the following system of equations and inequalities
$$
  h_i(x) = 0 \quad i \in \{ 1, \ldots, m \}, \quad g_k(x) \le 0 \quad k \in \{ 1, \ldots, l \}.
$$
Suppose that $Y$ is equipped with the Euclidean norm. Then, as it is easy to see, one has
$$
  d(y_0, \Phi(x)) = \sqrt{\sum_{i = 1}^m (h_i(x))^2 + \sum_{k = 1}^l \max\{ 0, g_k(x) \}^2}.
$$
Observe that in this case the penalty function $G_{\sigma}(x) = f(x) + \sigma d(y_0, \Phi(x))$ is exact if and only if
the standard $\ell_1$ penalty function
$$
  H_{\nu}(x) = f(x) + \nu \Big( \sum_{i = 1}^m |h_i(x)| + \sum_{k = 1}^l \max\{ 0, g_k(x) \} \Big)
$$
is exact. Furthermore, with the use of the well-known inequalities between the~Euclidean norm and the $\ell_1$ norm one
can easily show that the ex.p.p. $\sigma^*$ and $\nu^*$ of these functions satisfy the following
inequalities
$$
  \frac{1}{\sqrt{m + l}} \sigma^* \le \nu^* \le \sigma^*,
$$
and the same inequalities hold true for the local exact penalty parameters of these penalty functions. 

As a result, one obtains that in the case of equality and inequality constraints,
Theorems~\ref{Thrm_LocalExactEquiv_Simple} and \ref{Thrm_GlobalExactEquiv_Simple} describe a direct relation between the
exactness of the standard $\ell_1$ penalty function for a mathematical programming problem, and the exactness of 
the penalty function $F_{\lambda}(x, \varepsilon)$ for the same problem. Moreover, these theorems allow one to obtain
estimates of the (local or global) ex.p.p. of the penalty function $F_{\lambda}(x, \varepsilon)$ via the (local or
global) ex.p.p. of the $\ell_1$ penalty function.

Note that the results above correspond to the case $w = 0$ in (\ref{Def_SmoothExactPenFunc}). Let us show that the same
results can be obtained in the general case. We extend only Theorem~\ref{Thrm_GlobalExactEquiv_Simple} to a more
general case. Theorem~\ref{Thrm_LocalExactEquiv_Simple} can be extended in a similar way.

\begin{theorem} \label{Thrm_EquivExactness_NonSimple}
Let $Y$ be a normed space, $y_0 = 0$, and let for some $w \in Y$ one has
$$
  F_{\lambda}(x, \varepsilon) = 
  f(x) + \varepsilon^{-1} d(0, \Phi(x) - \varepsilon w)^2 + \lambda \varepsilon 
  \quad \forall \varepsilon > 0,
$$
where $d(0, \Phi(x) - \varepsilon w) = + \infty$ if $\Phi(x) = \emptyset$. Then the penalty function 
$F_{\lambda}(x, \varepsilon)$ is globally exact if and only if the penalty function 
$G_{\sigma}(x) = f(x) + \sigma d(0, \Phi(x))$ is globally exact, and
$$
  \frac{(\sigma^*)^2}{4} - \| w \| \sigma^* \le \lambda^* 
  \le \left( \frac{\sigma^*}{2} + \| w \| \right)^2,
$$
where $\sigma^*$ is the ex.p.p. of $G_{\sigma}$.
\end{theorem}

\begin{proof}
Suppose that $F_{\lambda}(x, \varepsilon)$ is exact. Then taking into account Remark~\ref{Rmrk_ExactPenCond} one gets
that
\begin{equation} \label{EquivDefOfExactness}
  \inf_{\varepsilon \ge 0} F_{\lambda}(x, \varepsilon) \ge \inf_{x \in \Omega} f(x) =: f^*
  \quad \forall x \in A \quad \forall \lambda \ge \lambda^*.
\end{equation}
Let us find an upper estimate of $\inf\{ F_{\lambda}(x, \varepsilon) \mid \varepsilon > 0 \}$. If $x \in \Omega$, then
$F_{\lambda}(x, \varepsilon) \le f(x) + (\lambda + \| w \|^2) \varepsilon$, which yields that the infimum is equal to
$f(x)$. If $x \in A \setminus \Omega$, then for any $\varepsilon > 0$ one has
$$
  F_{\lambda}(x, \varepsilon) = f(x) + \frac{1}{\varepsilon} d(0, \Phi(x) - \varepsilon w)^2 + \lambda \varepsilon.
$$
It is easy to see that $d(0, \Phi(x) - \varepsilon w) \le d(0, \Phi(x)) + \varepsilon \| w \|$. Hence one has
$$
  d(0, \Phi(x) - \varepsilon w)^2 \le d(0, \Phi(x))^2 + 2 \varepsilon \| w \| d(0, \Phi(x)) + \varepsilon^2 \| w \|^2,
$$
which implies that for any $x \in A$ and $\varepsilon > 0$ one has
$$
  F_{\lambda}(x, \varepsilon) \le f(x) + \frac{1}{\varepsilon} d(0, \Phi(x))^2 + 2 \| w \| d(0, \Phi(x)) + 
  (\lambda + \| w \|^2) \varepsilon.
$$
Minimizing the right-hand side of the latter inequality with respect to $\varepsilon$ one obtains that
$$
  \inf_{\varepsilon \ge 0} F_{\lambda}(x, \varepsilon) \le 
  f(x) + 2 (\sqrt{\lambda + \| w \|^2} + \| w \|) d(0, \Phi(x)) \quad \forall x \in A.
$$
Consequently, with the use of (\ref{EquivDefOfExactness}) one gets that
$$
  G_{\sigma}(x) = f(x) + \sigma d(0, \Phi(x)) \ge f^* \quad \forall x \in A \quad
  \forall \sigma \ge 2 (\sqrt{\lambda^* + \| w \|^2} + \| w \|).
$$
Hence the penalty function $G_{\sigma}$ is globally exact, and
$$
  \sigma^* \le 2(\sqrt{\lambda^* + \| w \|^2} + \| w \|) \iff
  \lambda^* \ge \frac{(\sigma^*)^2}{4} - \| w \| \sigma^*
$$
by virtue of Remark~\ref{Rmrk_ExactPenCond} (note that the expression on the right-hand side is negative, when
$\sigma^* < 4 \| w \|$). 

Suppose, now, that the penalty function $G_{\sigma}$ is globally exact. Then
\begin{equation} \label{StandPenFuncExactDef}
  G_{\sigma}(x) = f(x) + \sigma d(0, \Phi(x)) \ge f^* \quad \forall x \in A \quad \forall \sigma \ge \sigma^*.
\end{equation}
Let us find a lower estimate of $\inf\{ F_{\lambda}(x, \varepsilon) \mid \varepsilon \in \mathbb{R}_+ \}$ for any 
$x \in A \setminus \Omega$ (if $x \in \Omega$, then the infimum is equal to $f(x)$). 

Applying the well-known inequality $| \| w \| - \| v \| | \le \| w - v \|$ one obtains that
$$
  d(0, \Phi(x) - \varepsilon w) \ge d(0, \Phi(x)) - \varepsilon \| w \|.
$$
Hence one has that if $d(0, \Phi(x)) - \varepsilon \| w \| \ge 0$, then
$$
  d(0, \Phi(x) - \varepsilon w)^2 \ge d(0, \Phi(x))^2 - 2 \varepsilon \| w \| d(0, \Phi(x)) + 
  \varepsilon^2 \| w \|^2,
$$
while if $d(0, \Phi(x)) - \varepsilon \| w \| < 0$, then
$$
  d(0, \Phi(x) - \varepsilon w)^2 \ge 0 \ge d(0, \Phi(x)) \big( d(0, \Phi(x)) - 2 \varepsilon \| w \| \big).
$$
Consequently, one gets that
$$
  d(0, \Phi(x) - \varepsilon w)^2 \ge d(0, \Phi(x))^2 - 2 \varepsilon \| w \| d(0, \Phi(x)) \quad 
  \forall \varepsilon \in \mathbb{R}_+.
$$
Therefore for any $x \in A$ and $\varepsilon \ge 0$ one has
$$
  F_{\lambda}(x, \varepsilon) \ge f(x) + \frac{1}{\varepsilon} d(0, \Phi(x))^2 - 2 \| w \| d(0, \Phi(x)) +
  \lambda \varepsilon.
$$
Minimizing the right-hand side of the last inequality with respect to $\varepsilon$ one gets that
$$
  \inf_{\varepsilon \ge 0} F_{\lambda}(x, \varepsilon) \ge 
  f(x) + 2 \left( \sqrt{\lambda} - \| w \| \right) d(0, \Phi(x)).
$$
Hence applying (\ref{StandPenFuncExactDef}) one obtains that
$$
  \inf_{\varepsilon \ge 0} F_{\lambda}(x, \varepsilon) \ge f^* \quad \forall x \in A
$$
for any $\lambda \ge 0$ such that $2 (\sqrt{\lambda} - \| w \| ) \ge \sigma^*$. Thus, taking into account
Remark~\ref{Rmrk_ExactPenCond} one gets that the penalty function $F_{\lambda}(x, \varepsilon)$ is exact and
$$
  \lambda^* \le \left( \frac{\sigma^*}{2} + \| w \| \right)^2,
$$
that completes the proof.	 
\end{proof}

\begin{remark} \label{Rmrk_HowToReduceEXPP}
From the theorems above it follows that the ex.p.p. of the penalty function $F_{\lambda}(x, \varepsilon)$
asymptotically behaves like the square of the ex.p.p. of the standard nonsmooth exact penalty function 
$G_{\sigma}(x) = f(x) + \sigma d(y_0, \Phi(x))$. Thus, in the general case, the ex.p.p. of the penalty function
$F_{\lambda}(x, \varepsilon)$ is significantly larger, then the ex.p.p. of the standard exact penalty function. However,
one can easily modify this penalty function to reduce its ex.p.p. Namely, for some $\alpha > 0$ define the smooth
penalty function as follows
$$
  F_{\lambda}(x, \varepsilon) = f(x) + \frac{\lambda^{\alpha}}{\varepsilon} d(0, \Phi(x) - \varepsilon w)^2 + 
  \lambda \varepsilon \quad \forall \varepsilon > 0
$$
(cf. the penalty function in \cite{LianZhang}). It is easy to see that the ex.p.p. of this penalty function is
decreasing in $\alpha$. In particular, arguing in the same way as in the proof of
Theorem~\ref{Thrm_LocalExactEquiv_Simple} one can show that in the case $\alpha = 1$ and $w = 0$ one has
$\lambda^* = \sigma^* / 2$. Under the assumption that the mapping $\Phi$ is single valued, one can show that in the case
$\alpha = 1$ and $w \ne 0$ the following estimates hold true
$$
  \frac{\sigma^*}{2(\sqrt{1 + \|w\|^2} + \| w \|)} \le \lambda^* \le
  \frac{\sigma^*}{2(\sqrt{1 + \|w\|^2} - \| w \|)}
$$
(the assumption that $\Phi$ is single-valued allows one to use the more accurate lower estimate
$$
  d(0, \Phi(x) - \varepsilon w)^2 = \| \Phi(x) - \varepsilon w \|^2 \ge (\| \Phi(x) \| - \varepsilon \| w \|)^2
$$
in the proof of Theorem~\ref{Thrm_EquivExactness_NonSimple}).
\end{remark}

Theorem~\ref{Thrm_EquivExactness_NonSimple} provides estimates of the ex.p.p. of the penalty function $F_{\lambda}$
with arbitrary $w \in Y$. Let us show that a choice of $w$ can both increase and decrease the ex.p.p., and
that the lower estimate in Theorem~\ref{Thrm_EquivExactness_NonSimple} is sharp. 

\begin{example}
Let $X = A = \mathbb{R}^n$, $Y = \mathbb{R}$, $y_0 = 0$ and 
$$
  f(x) = x_1 + \ldots + x_n, \quad \Phi(x) = \big[ \| x \|^2 - 1, + \infty \big),
$$
where $\| \cdot \|$ is the Euclidean norm. Thus, the problem ($\mathcal{P}$) takes the form
\begin{equation} \label{example1}
  \min x_1 + \ldots + x_n \quad \text{subject to} \quad x_1^2 + \ldots + x_n^2 \le 1.
\end{equation}
It is easy to verify that that a unique point of global minimum of this problem has the form
$$
  x^* = \left( - \frac{1}{\sqrt{n}}, \ldots, - \frac{1}{\sqrt{n}} \right).
$$
Observe that the standard penalty function $G_{\sigma}(x) = f(x) + \sigma \max\{ 0, \| x \|^2 - 1 \}$ for problem
(\ref{example1}) is convex. Therefore $x^*$ is a point global minimum of $G_{\sigma}$ iff 
$0 \in \partial G_{\sigma}(x^*)$, where $\partial G_{\sigma}(x^*)$ is the subdifferential of $G_{\sigma}$ at $x^*$ in
the sense of convex analysis. From the fact that
$$
  \partial G_{\sigma} (x^*) = (1, \ldots, 1) + \co\{ 0, 2 \sigma x^* \},
$$
it follows that $0 \in \partial G_{\sigma}(x^*)$ iff $\sigma \ge \sqrt{n} / 2$. Hence the penalty function $G_{\sigma}$
is globally exact, and $\sigma^* = \sqrt{n} / 2$. Consequently, by Theorem~\ref{Thrm_EquivExactness_NonSimple} 
the penalty function $F_{\lambda}$ with arbitrary $w \in \mathbb{R}$ is also exact. Moreover, in the case $w = 0$ one
has $\lambda^* = n / 16$ by virtue of Theorem~\ref{Thrm_GlobalExactEquiv_Simple}. Let us compute the ex.p.p. of the
penalty function $F_{\lambda}$ in the case of arbitrary $w \in \mathbb{R}$.

Choose arbitrary $w < 0$. Then for any $\varepsilon > 0$ and $x \notin \Omega$ (i.e. $\| x \| > 1$) one has
$$
  F_{\lambda}(x, \varepsilon) = \sum_{i = 1}^n x_i + \frac{1}{\varepsilon} \max\{ 0, \| x \|^2 - 1 - \varepsilon w \}^2
  + \lambda \varepsilon.
$$
Note that $\| x \|^2 - 1 - \varepsilon w > 0$ for any $x \notin \Omega$ and $\varepsilon > 0$ due to the fact that 
$w < 0$. Hence for any such $x$ and $\varepsilon$ one has
$$
  F_{\lambda}(x, \varepsilon) = \sum_{i = 1}^n x_i + \frac{1}{\varepsilon} ( \| x \|^2 - 1 - \varepsilon w )^2 + 
  \lambda \varepsilon.
$$
Minimizing the right-hand side with respect to $\varepsilon > 0$ one obtains
$$
  \min_{\varepsilon > 0} F_{\lambda}(x, \varepsilon) = 
  \sum_{i = 1}^n x_i + 2 (\sqrt{\lambda + w^2} - w) ( \| x \|^2 - 1 ) = G_{\sigma}(x),
$$
where $\sigma = 2 (\sqrt{\lambda + w^2} - w)$. Consequently, one has that
$$
  \min_{\varepsilon > 0} F_{\lambda}(x, \varepsilon) = G_{\sigma}(x) \ge f^* = f(x^*) \quad \forall x \notin \Omega
$$
if and only if $\sigma \ge \sigma^* = \sqrt{n} / 2$. Therefore taking into account Remark~\ref{Rmrk_ExactPenCond} one
obtains that $\lambda^*$ is equal to the greatest lower bound of all $\lambda \ge 0$ for which 
$2 (\sqrt{\lambda + w^2} - w) \ge \sqrt{n} / 2$. Hence
$$
  \lambda^* = \begin{cases}
    0, & \text{if } |w| \ge \sqrt{n} / 8, \\
    n / 16 - |w| \sqrt{n} / 2, & \text{otherwise},
  \end{cases}
$$
or, equivalently, $\lambda^* = \max\{ 0, (\sigma^*)^2 / 4 - |w| \sigma^* \}$. Thus, the lower estimate in
Theorem~\ref{Thrm_EquivExactness_NonSimple} is sharp. Note also that in the case $w < 0$ the ex.p.p. of $F_{\lambda}$
is smaller than in the case $w = 0$.

Let, now, $w > 0$ be arbitrary. Fix $x \in \mathbb{R}^n$ such that $\| x \| > 1$, and denote
$E = \{ \varepsilon > 0 \mid \| x \|^2 - 1 - \varepsilon w > 0 \}$. Clearly, $E = (0, (\| x \|^2 - 1) / w)$. 
For any $\varepsilon \in E$ one has
$$
  F_{\lambda}(x, \varepsilon) = \sum_{i = 1}^n x_i + \frac{1}{\varepsilon} ( \| x \|^2 - 1 - \varepsilon w )^2 + 
  \lambda \varepsilon.
$$
Introduce the function 
$$
  h(\varepsilon) = \frac{1}{\varepsilon} ( \| x \|^2 - 1 - \varepsilon w )^2 + \lambda \varepsilon.
$$
Let us find a global minimum of the function $h$ on the set $E$. Solving the equation $h'(\varepsilon^*) = 0$ one gets
$\varepsilon^* = (\| x \|^2 - 1) / \sqrt{\lambda + w^2}$. It is easy to check that $\varepsilon^* \in E$, 
$h'(\varepsilon) < 0$ for any $\varepsilon \in (0, \varepsilon^*)$ and $h'(\varepsilon) > 0$ for any 
$\varepsilon > \varepsilon^*$. Therefore $\varepsilon^*$ is a point of global minimum of the function $h$ on the set
$E$. Hence
$$
  \min_{\varepsilon \in E} F_{\lambda}(x, \varepsilon) = \sum_{i = 1}^n x_i + h(\varepsilon^*) = 
  \sum_{i = 1}^n x_i + 2 (\sqrt{\lambda + w^2} - w) ( \| x \|^2 - 1 ) = G_{\sigma}(x),
$$
where $\sigma = 2(\sqrt{\lambda + w^2} - w)$. On the other hand, if $\varepsilon \notin E$, i.e. if 
$\varepsilon \ge (\| x \|^2 - 1) / w$, then
$$
  F_{\lambda}(x, \varepsilon) = \sum_{i = 1}^n x_i + \lambda \varepsilon \ge
  \sum_{i = 1}^n x_i + \frac{\lambda}{w} (\| x \|^2 - 1) = G_{\lambda / w}(x).
$$
Therefore
$$
  \min_{\varepsilon > 0} F_{\lambda}(x, \varepsilon) = G_{\gamma}(x) \quad \forall x \notin \Omega,
$$
where $\gamma = \min\{ \lambda / w, 2 (\sqrt{\lambda + w^2} - w) \}$. Then arguing in the same way as in the
case $w < 0$ one obtains that $\lambda^*$ coincides with the greatest lower bound of all $\lambda \ge 0$ for which
$\gamma \ge \sigma^*$, which yields
$$
  \lambda^* = \frac{(\sigma^*)^2}{4} + w \sigma^* = \frac{n}{16} + w \frac{\sqrt{n}}{2}.
$$
Note that in the case $w > 0$ the ex.p.p. of $F_{\lambda}$ is greater than in the case $w = 0$.
\end{example}

\begin{remark}
If in Theorem~\ref{Thrm_EquivExactness_NonSimple} the mapping $\Phi$ is single-valued or it has the form 
$\Phi(x) = [g(x), + \infty)$ for some function $g \colon X \to \mathbb{R}$, then one can obtain the more accurate upper
estimate $\lambda^* \le (\sigma^*)^2 / 4 + \| w \| \sigma^*$ (cf.~Remark~\ref{Rmrk_HowToReduceEXPP}). Furthermore, as
the example above shows, this upper estimate is sharp.
\end{remark}

\section{Nonlinear Trasformations of Smooth Penalty Functions}
\label{NonlinearCase}

In this section, we study how the introduction of nonlinear functions $\phi$ and $\beta$ into the definition of 
the smooth penalty function (see~(\ref{Def_SmoothExactPenFunc})) affects its exactness.

\subsection{The Case $w = 0$}

Let $\phi \colon [0, + \infty] \to [0, +\infty]$ be a nondecreasing function such that $\phi(t) = 0$ iff $t = 0$
(the element $+\infty$ is included into the domain of $\phi$ in order to allow $\Phi(x)$ to be empty for some $x$).
Introduce the following penalty function
$$
  F_{\lambda}[\phi](x, \varepsilon) = 
  f(x) + \varepsilon^{-1} \phi(d(y_0, \Phi(x))^2) + \lambda \varepsilon
  \quad \forall \varepsilon > 0.
$$
If $\phi(t) \equiv t$, then we simply write $F_{\lambda}(x, \varepsilon)$. In order to underline the effect of the
function $\phi$, denote the ex.p.p. of this penalty function at a point  $x^* \in \Omega$ by 
$\lambda^*(x^*, \phi)$. We will also use the similar notation for the global ex.p.p.

At first, note that Theorems~\ref{Thrm_LocalExactEquiv_Simple} and \ref{Thrm_GlobalExactEquiv_Simple} can be easily
generalized to the case of the penalty function above. In particular, the following result holds true.

\begin{theorem} \label{Thrm_GlobalExactEquiv_Nonlinear}
The penalty function $F_{\lambda}[\phi](x, \varepsilon)$ is globally exact if and only if the penalty function 
$G_{\sigma}[\phi](x) = f(x) + \sigma \sqrt{\phi( d(y_0, \Phi(x))^2 )}$ is globally exact, and 
$\lambda^*(\phi) = \sigma^*(\phi)^2 / 4$, where $\sigma^*(\phi)$ is the ex.p.p. of $G_{\sigma}[\phi]$.
\end{theorem}

\begin{remark}
Let the function $\phi$ be twice continuously differentiable on $[0, t_0]$ for some $t_0 > 0$. From the theorem above it
follows that for the penalty function $F_{\lambda}[\phi]$ to be exact in the general case it is necessary that
$\phi'(0) > 0$. Indeed, let $X$ be a normed space, $A = X$, and let the functions $f$ and $d(y_0, \Phi(\cdot))^2$ be
G\^ateaux differentiable at a globally optimal solution $x^*$ of the problem ($\mathcal{P}$). Arguing by reductio ad
absrudum, suppose that $\phi'(0) = 0$ (note that since $\phi$ is nondecreasing, then $\phi'(0) \ge 0$), but the penalty
function $F_{\lambda}[\phi]$ is globally exact. Then by Theorem~\ref{Thrm_GlobalExactEquiv_Nonlinear} the point $x^*$ is
a point of global minimum of the penalty function $G_{\sigma}$. 

Let us show that the function $G_{\sigma}$ is G\^ateaux differentiable at $x^*$. For any $h \in X$ and 
$\alpha \in \mathbb{R}$ denote $\omega(\alpha) = d(y_0, \Phi(x^* + \alpha h))^2$. From the fact that the function 
$d(y_0, \Phi(\cdot))^2$ is G\^ateaux differentiable at $x^*$, and $x^*$ is a point global minimum of this function
(recall that $d(y_0, \Phi(\cdot))^2$ is nonnegative, and $d(y_0, \Phi(x^*))^2 = 0$) it follows that $\omega$ is
differentiable at $0$, and $\omega'(0) = 0$. Applying the Taylor expansion for the function $\phi$ at $0$ one obtains
that for any sufficiently small $\alpha > 0$ there exists $\tau \in [0, \omega(\alpha)]$ such that
$$
  \frac{1}{\alpha} \sqrt{\phi( d(y_0, \Phi(x^* + \alpha h))^2 )} = 
  \frac{1}{\alpha} \sqrt{\phi\big( \omega(\alpha) \big)} =
  \sqrt{\frac{1}{\alpha^2} \frac{\phi''(\tau)}{2} \omega(\alpha)^2 }.
$$
Passing to the limit as $\alpha \to +0$ one gets that
$$
  \frac{d}{d \alpha} \sqrt{\phi( d(y_0, \Phi(x^* + \alpha h))^2 )} =  \sqrt{\frac{\phi''(0)}{2} \omega'(0)^2} = 0.
$$
Hence the function $\sqrt{\phi( d(y_0, \Phi(\cdot))^2 )}$ is G\^ateaux differentiable at $x^*$, and its G\^ateaux
derivative is equal to $0$. Therefore the function $G_{\sigma}$ is also G\^ateaux differentiable at $x^*$, and
$G'_{\sigma}(x^*) = f'(x^*)$, which implies $f'(x^*) = 0$ due to the fact that $x^*$ is a point of global minimum of
$G_{\sigma}$. However, in the general case the equality $f'(x^*) = 0$ does not hold true, since $x^*$ is a point of
global minimum of the \textit{constrained} optimization problem ($\mathcal{P}$). Thus, in the general case, for the
penalty function $F_{\lambda}[\phi]$ to be exact it is necessary that $\phi'(0) > 0$. In particular, for any 
$\theta > 0$ the penalty function
$$
  F_{\lambda}(x, \varepsilon) = f(x) + \frac{1}{\varepsilon} d(y_0, \Phi(x))^{2 + \theta} + \lambda \varepsilon
$$
is not exact, provided there exists a point of global minimum $x^*$ of the problem ($\mathcal{P}$) such that 
$f'(x^*) \ne 0$.
\end{remark}

Let us study how the exactness of the penalty function $F_{\lambda}[\phi]$ changes with respect to a change of the
function $\phi$. We start we the case of local exactness.

\begin{theorem} \label{Thrm_LocalExactTrasform_Nonlinear}
Let $\psi \colon [0, + \infty] \to [0, +\infty]$ be a nondecreasing function such that $\psi(t) = 0$ iff $t = 0$. Let
also $x^* \in \dom f$ be a locally optimal solution of the problem ($\mathcal{P}$). Suppose that the following
assumptions hold true:
\begin{enumerate}
\item{the penalty function $F_{\lambda}[\phi]$ is exact at $x^*$;}

\item{there exist $\psi_0 > 0$ and $t_0 > 0$ such that $\psi(t) \ge \psi_0 \phi(t)$ for all $t \in [0, t_0]$;
}

\item{the function $d(y_0, \Phi(\cdot))$ is continuous at $x^*$.
}
\end{enumerate}
Then the penalty function $F_{\lambda}[\psi]$ is also exact at $x^*$, and
$$
  \lambda^*(x^*, \psi) \le \frac{\lambda^*(x^*, \phi)}{\psi_0}
$$
\end{theorem}

\begin{proof}
The mapping $d(y_0, \Phi(\cdot))$ is continuous at $x^*$ and $d(y_0, \Phi(x^*)) = 0$ by the fact that $x^*$ is
feasible. Therefore there exists a neighbourhood $U$ of $x^*$ such that
\begin{equation} \label{ContinNearLocalMin_NonlinCase}
  d(y_0, \Phi(x))^2 \le t_0 \quad \forall x \in U.
\end{equation}
Taking into account the fact that $F_{\lambda}[\phi]$ is exact at $x^*$ one obtains that for any 
$\lambda > \lambda^*(x^*, \phi)$ there exist a neighbourhood $V \subset U$ of $x^*$ and $\varepsilon_0 > 0$ for which
$$
  F_{\lambda}[\phi](x, \varepsilon) \ge F_{\lambda}[\phi](x^*, 0) = f(x^*) 
  \quad \forall (x, \varepsilon) \in V \times [0, \varepsilon_0].
$$
Consider now the penalty function $F_{\lambda}[\psi]$. If $\varepsilon = 0$, then
$$
  F_{\lambda}[\psi](x, 0) = F_{\lambda}[\phi](x, 0) \ge f(x^*) = F_{\lambda}[\psi](x^*, 0) 
  \quad \forall x \in V.
$$
On the other hand, if $\varepsilon \in (0, \varepsilon_0]$, then applying the inequality 
$\psi(t) \ge \psi_0 \phi(t)$, and taking into account (\ref{ContinNearLocalMin_NonlinCase}) one gets that
\begin{multline*}
  f(x^*) \le F_{\lambda}[\phi](x, \varepsilon) = 
  f(x) + \frac{1}{\varepsilon} \phi( d(y_0, \Phi(x))^2 ) + \lambda \varepsilon \le \\
  \le f(x) + \frac{1}{\varepsilon \psi_0} \psi( d(y_0, \Phi(x))^2 ) + \lambda \varepsilon =
  F_{\lambda / \psi_0}[\psi](x, \varepsilon \psi_0).
\end{multline*}
for any $x \in V$. Therefore for any $\lambda > \lambda^*(x^*, \phi)$ one has
$$
  F_{\lambda / \psi_0}[\psi](x, \varepsilon) \ge f(x^*) \quad 
  \forall (x, \varepsilon) \in V \times [0, \varepsilon_0 / \psi_0],  
$$
which implies that the penalty function $F_{\lambda}[\psi]$ is exact at $x^*$ and 
$\lambda^*(x^*, \psi) \le \lambda^*(x^*, \phi) / \psi_0$.	 
\end{proof}

\begin{corollary} \label{Crlr_LocExPP_Nonlin_Nonshifted}
Let $x^* \in \dom f$ be a locally optimal solution of the problem ($\mathcal{P}$), and the mapping $d(y_0, \Phi(\cdot))$
be continuous at $x^*$. Suppose that there exists the right-hand side derivative $\phi'_+(0)$ of $\phi$ at $0$ such that
$\phi'_+(0) > 0$. Then the penalty function $F_{\lambda}[\phi]$ is exact at $x^*$ if and only if the penalty function
$F_{\lambda}$ is exact at this point and
$$
  \lambda^*(x^*, \phi) = \frac{\lambda^*(x^*)}{\phi'_+(0)}.
$$
\end{corollary}

The previous corollary can be partly generalized to the case of global exactness.

\begin{theorem} \label{Theorem_Convexity}
Let the penalty function $F_{\lambda}$ be exact. Suppose also that $\phi$ is convex, and there exists the right-hand
side derivative $\phi'_+(0)$ of $\phi$ at $0$ such that $\phi'_+(0) > 0$. Then the penalty function $F_{\lambda}[\phi]$
is exact and
$$
  \lambda^*(\phi) \le \frac{\lambda^*}{\phi'_+(0)}.
$$
\end{theorem}

\begin{proof}
Since the function $\phi$ is convex, then $\phi(t) \ge \phi'_+(0) t$ for all $t \ge 0$ 
(\cite{Rockafellar}, Theorem 23.1). Therefore one has
$$
  F_{\lambda}[\phi](x, \varepsilon) \ge F_{\lambda \phi'_+(0)}\big( x, \varepsilon / \phi'_+(0) \big) \quad
  \forall (x, \varepsilon) \in X \times \mathbb{R}_+,
$$
which implies the desired result.	 
\end{proof}

The theorem above provides only an upper estimate of the ex.p.p. of the penalty function $F_{\lambda}[\phi]$.
Furthermore, this estimate relies primarily on the information about the behaviour of a function $\phi$ in a
neighbourhood of zero (namely, it depends only on $\phi'_+(0)$), while a possible effect of the nonlinearity of the
function $\phi$ for large values of the constraint violation measure $d(y_0, \Phi(x))^2$ is not taken into account
explicitly. Let us show that on one hand, the estimate of the ex.p.p. in Theorem~\ref{Theorem_Convexity} is sharp,
but on the other hand this estimate is very crude, since even in the case $\phi'_+(0) = 1$ the ex.p.p. $\lambda^*(\phi)$
can be significantly smaller than the ex.p.p. $\lambda^*$.

\begin{example}
Let $X = A = Y = \mathbb{R}$, $y_0 = 0$ and $\Phi(x) = [ x, + \infty )$. For any $c \ge 0$ define
$$
  f(x) = \begin{cases}
    -x, & \text{if } x \le 1, \\
    - 0.5 x^2 - 0.5, & \text{if } x \in (1, c + 1), \\
    - (c + 1) x + 0.5 c^2 + c, & \text{if } x \ge c + 1.
  \end{cases}
$$
It is easy to verify that the function $f$ is continuously differentiable.

Since the inclusion $0 \in \Phi(x)$ is equivalent to the inequality $x \le 0$, then the problem ($\mathcal{P}$)
is equivalent to the problem of minimizing the function $f$ over the set $(- \infty, 0]$. Clearly, a unique globally
optimal solution of this problem is the point $x^* = 0$ and $f^* = 0$. 

The standard penalty function for the problem ($\mathcal{P}$) has the form 
$G_{\sigma}(x) = f(x) + \sigma \max\{ 0, x \}$. Note that
$$
  G_{\sigma}(x) = - (c + 1) x + 0.5 c^2 + c + \sigma x \ge f^* = 0 \quad \forall x \ge c + 1
$$
if and only if $\sigma \ge c + 1$. Moreover, since $G'_{\sigma}(x) = -x + \sigma$ for any $x \in (1, c + 1)$, then
$G_{\sigma}(x) \ge f^*$ for any $x \in (1, c + 1)$ and $\sigma \ge c + 1$. Therefore
$G_{\sigma}(x) \ge f^*$ for all $x \in \mathbb{R}$ if and only if $\sigma \ge c + 1$. Hence the penalty function
$G_{\sigma}$ is globally exact and $\sigma^* = c + 1$. Consequently, the smooth penalty function $F_{\lambda}$ 
(with $w = 0$) is also exact, and $\lambda^* = (c + 1)^2 / 4$ by virtue of Theorem~\ref{Thrm_GlobalExactEquiv_Simple}.

Define
$$
  \phi(t) = \begin{cases}
    \dfrac{t}{1 - t}, & \text{if } t \in [0, 1), \\
    + \infty, & \text{if } t \ge 1,
  \end{cases}
$$
and introduce the penalty function $G_{\sigma}[\phi](x) = f(x) + \sqrt{\phi( d(0, \Phi(x))^2 )}$. Observe that
$\phi(t) \ge t$ for all $t \in [0, 1)$ and $d(0, \Phi(x))^2 = \max\{ 0, x \}^2$. Therefore
$$
  G_{\sigma}[\phi](x) = - x + \sqrt{ \phi\big( \max\{ 0, x \}^2 \big) } \ge - x + \sigma \max\{ 0, x \} 
  \quad \forall	x \in [0, 1),
$$
and $G_{\sigma}[\phi](x) = + \infty$ for any $x \ge 1$. Hence for any $\sigma \ge 1$ one has 
$G_{\sigma}[\phi](x) \ge 0 = f^*$ for all $x \in \mathbb{R}$, which implies that the penalty function
$G_{\sigma}[\phi]$ is exact and $\sigma^*(\phi) \le 1$.

Let us show that $\sigma^*(\phi) = 1$. Indeed, let $\sigma \in (0, 1)$. Then there exists $\varepsilon > 0$ such that
$\sigma + \varepsilon < 1$. Hence for $x = \sqrt{1 - (\sigma + \varepsilon)^2}$ one has
$$
  G_{\sigma}[\phi](x) = - x + \sigma \sqrt{\phi(x^2)} = - x + \sigma \frac{x}{\sigma + \varepsilon} < 0 = f^*.
$$
Therefore $x^* = 0$ is not a point of global minimum of the penalty function $G_{\sigma}[\phi]$ for any 
$\sigma \in (0, 1)$, which yields $\sigma^*(\phi) = 1$. Applying Theorem~\ref{Thrm_GlobalExactEquiv_Nonlinear} one
gets that the penalty function $F_{\lambda}[\phi]$ is globally exact and $\lambda^*(\phi) = 1 / 4$.

Since $\phi'_+(0) = 1$, then with the use of Theorem~\ref{Theorem_Convexity} one obtains the estimate 
$\lambda^*(\phi) \le \lambda^* = (c + 1)^2 / 4$, that turns into an equality in the case $c = 0$. Thus, this estimate is
sharp. However, note also that $\lambda^* \to + \infty$ as $c \to \infty$, while $\lambda^*(\phi) = 1 / 4$ for all $c$.
\end{example}

The proof of Theorem~\ref{Theorem_Convexity} essentially relies on the convexity of the function $\phi$. Let us
show that a more sophisticated argument allows one to avoid this assumption. However, it should be underlined that this
result does not contain any estimates of the exact penalty parameter.

We need the following auxiliary result.

\begin{lemma} \label{LemmaExPenFunc}
The penalty function $G_{\sigma}(x) = f(x) + \sigma \phi(d(y_0, \Phi(x))^2)$ is globally exact if and only if the
function $G_{\sigma_1}$ is bounded below on $A$ for some $\sigma_1 > 0$, and there exists $\delta > 0$ such that
\begin{equation} \label{ExactOnOmegaDelta}
  G_{\sigma_2}(x) \ge f^* \quad \forall x \in \Omega_{\delta} = 
  \big\{ z \in A \mid \phi(d(y_0, \Phi(z))^2) < \delta \big\},
\end{equation}
for some $\sigma_2 > 0$, where $f^* = \inf_{x \in \Omega} f(x)$.
\end{lemma}

\begin{proof}
If the penalty function $G_{\sigma}$ is globally exact, then, obviously, for any $\sigma > \sigma^*$ the function
$G_{\sigma}$ is bounded below on $A$, and the condition (\ref{ExactOnOmegaDelta}) is valid for any $\delta > 0$.
Let us prove the converse statement. For any $x \in A \setminus \Omega_{\delta}$ one has
\begin{multline*}
  G_{\sigma}(x) = f(x) + \sigma \phi(d(y_0, \Phi(x))^2) = 
  G_{\sigma_1}(x) + (\sigma - \sigma_1) \phi(d(y_0, \Phi(x))^2) \ge \\
  \ge c + (\sigma - \sigma_1) \delta \ge f^*  \quad \forall \sigma \ge \widehat{\sigma}
\end{multline*}
where
$$
  c = \inf_{x \in A} G_{\sigma_1}(x) > - \infty, \quad 
  \widehat{\sigma} = \sigma_1 + \frac{f^* - c}{\delta}.
$$
Therefore $G_{\sigma}(x) \ge f^*$ for all $x \in A$ and $\sigma \ge \max\{ \sigma_2, \widehat{\sigma} \}$. It remains
to apply Remark~\ref{Rmrk_ExactPenCond}.	 
\end{proof}

\begin{theorem}
Let $\psi \colon [0, + \infty] \to [0, +\infty]$ be a nondecreasing function such that $\psi(t) = 0$ iff $t = 0$.
Suppose that following assumptions hold true:
\begin{enumerate}
\item{the penalty function $F_{\lambda}[\phi]$ is globally exact;}

\item{there exist $\psi_0 > 0$ and $t_0 > 0$ such that $\psi(t) \ge \psi_0 \phi(t)$ for all $t \in [0, t_0]$;
}

\item{the function $H_{\sigma}(x) = f(x) + \sigma \sqrt{\psi(d(y_0, \Phi(x))^2)}$ is bounded below on $A$ 
for some $\sigma \ge 0$.
}
\end{enumerate}
Then the penalty function $F_{\lambda}[\psi]$ is globally exact.
\end{theorem}

\begin{proof}
From the fact the penalty function $F_{\lambda}[\phi]$ is globally exact, and
Theorem~\ref{Thrm_GlobalExactEquiv_Nonlinear} it follows that the
penalty function $G_{\sigma}(x) = f(x) + \sigma \sqrt{\psi(d(y_0, \Phi(x))^2)}$ is also globally exact. Therefore there
exists $\sigma > 0$ such that
$$
  G_{\sigma}(x) \ge f^* := \inf_{x \in \Omega} f(x) \quad \forall x \in A.
$$
Denote $\delta = \sqrt{\psi(t_0)}$. Applying the inequality $\psi(t) \ge \psi_0 \phi(t)$, and taking into
account the fact that the function $\psi$ is nondecreasing one gets that
$$
  f^* \le G_{\sigma}(x) \le H_{\sigma / \sqrt{\psi_0}}(x) \quad
  \forall x \in \big\{ z \in A \mid \sqrt{\psi(d(y_0, \Phi(z))^2)} < \delta \big\}
$$
Hence and from Lemma~\ref{LemmaExPenFunc} it follows that the penalty function $H_{\sigma}$ is exact. Then applying
Theorem~\ref{Thrm_GlobalExactEquiv_Nonlinear} one obtains the required result.	 
\end{proof}

\subsection{The General Case}

Let, now, $Y$ be a normed space, $y_0 = 0$, and let $\phi \colon [0, + \infty] \to [0, +\infty]$ be a nondecreasing
function such that $\phi(t) = 0$ iff $t = 0$. For any $w \in Y$ define
\begin{equation} \label{NonlinearShiftedPenFunc}
  F_{\lambda}[\phi, w](x, \varepsilon) = 
    f(x) + \varepsilon^{-1} \phi(d(0, \Phi(x) - \varepsilon w)^2) + \lambda \varepsilon 
    \quad \forall \varepsilon > 0.
\end{equation}
If $\phi(t) \equiv t$, then we write $F_{\lambda}[w](x, \varepsilon)$. Denote the ex.p. p. of this function at
a point $x^* \in \Omega$ by $\lambda^*(x^*, \phi, w)$, and denote its global ex.p.p. by $\lambda^*(\phi, w)$.

Let us show that the results of the previous subsection cannot be directly generalized to the case of arbitrary 
$w \in Y$.

\begin{example}
Let $X = Y = A = \mathbb{R}$, $\Phi(x) = x$, $f(x) = - \sign(x) \sqrt{|x|}$, and $\phi(t) = \sqrt{t}$. Observe that
$\Omega = \{ 0 \}$ and
$$
  G_{\sigma}(x) = f(x) + \sigma \sqrt{ \phi( d(0, \Phi(x))^2 ) } = - \sign(x) \sqrt{|x|} + \sigma \sqrt{|x|}.
$$
Therefore the penalty function $G_{\sigma}(x)$ is exact, and $\sigma^*(\phi) = 1$. Hence by 
Theorem~\ref{Thrm_GlobalExactEquiv_Nonlinear} the penalty function $F_{\lambda}[\phi] = F_{\lambda}[\phi, 0]$ is
also exact, and $\lambda^*(\phi) = 1 / 4$.

Let now $w > 0$ be arbitrary. Then for any $\varepsilon > 0$ and $\lambda > 0$ one has
$$
  F_{\lambda}[\phi, w] ( \varepsilon w, \varepsilon ) = - \sign(\varepsilon w) \sqrt{|\varepsilon w|} + 
  \frac{1}{\varepsilon} | \varepsilon w - \varepsilon w | + \lambda \varepsilon = 
  - \sqrt{ \varepsilon w } + \lambda \varepsilon,
$$
which yields $F_{\lambda}[\phi, w] ( \varepsilon w, \varepsilon ) < 0 = f(0)$ for any sufficiently small 
$\varepsilon > 0$, and any $\lambda > 0$. Thus, the penalty function $F_{\lambda}[\phi, w]$ is not exact 
for any $w > 0$.
\end{example}

However, Theorems~\ref{Thrm_LocalExactTrasform_Nonlinear} and \ref{Theorem_Convexity} can be partly extended to the
general case.

\begin{theorem}
Let $\psi \colon [0, + \infty] \to [0, +\infty]$ be a nondecreasing function such that $\psi(t) = 0$ iff $t = 0$. Let
also $x^* \in \dom f$ be a locally optimal solution of the problem $(\mathcal{P})$. Suppose that following assumptions
hold true:
\begin{enumerate}
\item{there exist $\psi_0 > 0$ and $t_0 > 0$ such that $\psi(t) \ge \psi_0 \phi(t)$ for all $t \in [0, t_0]$;
}

\item{the penalty function $F_{\lambda}[\phi, \psi_0 w]$ is exact at $x^*$;}

\item{the function $(x, \varepsilon) \to d(0, \Phi(x) - \varepsilon w)$ is continuous at $(x^*, 0)$.
}
\end{enumerate}
Then the penalty function $F_{\lambda}[\psi, w]$ is also exact at $x^*$ and
$$
  \lambda^*(x^*, \psi, w) \le \frac{\lambda^*(x^*, \phi, \psi_0 w)}{\psi_0}
$$
\end{theorem}

\begin{proof}
From the inequality $\psi(t) \ge \psi_0 \phi(t)$ it follows that for any $x \in A$ and $\varepsilon > 0$ such that
$d(0, \Phi(x) - \varepsilon w)^2 \le t_0$ one has
\begin{multline*}
  F_{\lambda}[\psi, w] (x, \varepsilon) \ge f(x) + 
  \frac{\psi_0}{\varepsilon} \phi( d(0, \Phi(x) - \varepsilon w)^2 ) + \lambda \psi_0 \frac{\varepsilon}{\psi_0} = \\
  = F_{\lambda \psi_0}[\phi, \psi_0 w] \left( x, \frac{\varepsilon}{\psi_0} \right).
\end{multline*}
Then applying the continuity of the mapping $(x, \varepsilon) \to d(0, \Phi(x) - \varepsilon w)$, and arguing in the
same way as in the proof of Theorem~\ref{Thrm_LocalExactTrasform_Nonlinear} one obtains the desired result.	 
\end{proof}

\begin{corollary}
Let $x^* \in \dom f$ be a locally optimal solution of the problem ($\mathcal{P}$), and the mapping 
$(x, \varepsilon) \to d(0, \Phi(x) - \varepsilon w)$ be continuous at $(x^*, 0)$. Suppose that there exists the
right-hand side derivative $\phi'_+(0)$ of $\phi$ at $0$ such that $\phi'_+(0) > 0$. Then the penalty function
$F_{\lambda}[\phi, w]$ is exact at $x^*$ if and only if the penalty function $F_{\lambda}[\phi'_+(0) w]$ is exact
at this point, and for any $0 < \phi_1 < \phi'_+(0) < \phi_2$ one has 
$\lambda^*(x^*, \phi_2 w) / \phi_2 \le \lambda^*(x^*, \phi, w) \le \lambda^*(x^*, \phi_1 w) / \phi_1$.
\end{corollary}

As in the case $w = 0$, the corollary above can be extended to the case of global exactness under the assumption that
the function $\phi$ is convex.

\begin{theorem}
Let $\phi \colon [0, + \infty] \to [0, +\infty]$ be a nondecreasing convex function such that $\phi(t) = 0$ iff
$t = 0$, and let there exist the right-hand side derivative $\phi'_+(0)$ of $\phi$ at $0$ such that
$\phi'_+(0) > 0$. Suppose also that the penalty function $F_{\lambda}[\phi'_+(0) w]$ is exact. Then the penalty
function $F_{\lambda} [\phi, w]$ is also exact and 
$\lambda^*(\phi, w) \le \lambda^*(\phi'_+(0) w) / \phi'_+(0)$.
\end{theorem}

\subsection{Nonlinear Dependence on $\varepsilon$}

Let, as above, $Y$ be a normed space, $y_0 = 0$, and let $\phi \colon [0, + \infty] \to [0, +\infty]$ and
$\beta \colon [0, + \infty) \to [0, + \infty]$ be nondecreasing functions such that $\phi(t) = 0$ iff $t = 0$, and 
$\beta(t) = 0$ iff $t = 0$. For any $w \in Y$ define the penalty function
$$
  F_{\lambda}[\phi, w, \beta](x, \varepsilon) = 
    f(x) + \varepsilon^{-1} \phi(d(0, \Phi(x) - \varepsilon w)^2) + \lambda \beta(\varepsilon)
    \quad \forall \varepsilon > 0.
$$
Denote the ex.p. p. of this function at a point $x^* \in \Omega$ by $\lambda^*(x^*, \phi, w, \beta)$, and denote its
global ex.p.p. by $\lambda^*(\phi, w, \beta)$.

Let us show that under some natural assumptions the case of nonlinear function $\beta$ can be easily reduced to the
case $\beta(t) \equiv t$.

\begin{theorem}
Let $x^* \in \dom f$ be a locally optimal solution of the problem ($\mathcal{P}$), and let there exist the right-hand
side derivative $\beta'_+(0)$ of $\beta$ at $0$ such that $\beta'_+(0) > 0$. Then the penalty function 
$F_{\lambda}[\phi, w, \beta]$ is exact at $x^*$ if and only if the penalty function $F_{\lambda}[\phi, w]$ is
exact at this point, and $\lambda^*(\phi, w, \beta) = \lambda^*(\phi, w) / \beta'_+(0)$.
\end{theorem}

\begin{proof}
Since $\beta'_+(0) > 0$, for any $\eta \in (0, \beta'_+(0))$ there exists $\varepsilon_0 > 0$ such that 
$$
  0 \le ( \beta'_+(0) - \eta ) \varepsilon \le \beta(\varepsilon) \le ( \beta'_+(0) + \eta ) \varepsilon \quad
  \forall \varepsilon \in [0, \varepsilon_0),
$$
which yields that for any $(x, \varepsilon) \in A \times [0, \varepsilon_0)$ one has
$$
  F_{(\beta'_+(0) - \eta) \lambda}[\phi, w](x, \varepsilon) \le F_{\lambda}[\phi, w, \beta](x, \varepsilon) \le
  F_{(\beta'_+(0) + \eta) \lambda}[\phi, w](x, \varepsilon).
$$
Therefore $F_{\lambda}[\phi, w, \beta]$ is exact at $x^*$ iff $F_{\lambda}[\phi, w]$ is exact at this point, and
$\lambda^*(x^*, \phi, w, \beta) = \lambda^*(x^*, \phi, w) / \beta'_+(0)$ due to the fact that 
$\eta \in (0, \beta'_+(0))$ was chosen arbitrarily.  
\end{proof}

The previous theorem can be extended to the case of global exactness under the assumption that the function $\beta$
is convex. In this case, one easily obtains the estimate 
$\lambda^*(\phi, w, \beta) \le \lambda^*(\phi, w) / \beta'_+(0)$. However, as in the case of 
Theorem~\ref{Theorem_Convexity}, the convexity assumption can be discarded.

Arguing in a similar way to the proof of Lemma~\ref{LemmaExPenFunc} one can verify that the following result holds true.

\begin{lemma} \label{Lemma_EpsReduction}
The penalty function $F_{\lambda}[\phi, w, \beta]$ is globally exact if and only if there exists $\lambda_1 \ge 0$ such
that the function $F_{\lambda_1}[\phi, w, \beta]$ is bounded below on $A \times \mathbb{R}_+$, and there exists
$\varepsilon_0 > 0$ such that
$$
  F_{\lambda_2}[\phi, w, \beta] \ge f^* \quad \forall (x, \varepsilon) \in A \times [0, \varepsilon_0)
$$
for some $\lambda_2 \ge 0$, where $f^* = \inf_{x \in \Omega} f(x)$.
\end{lemma}

\begin{theorem} \label{Theorem_NonlinearDepnOnEps}
Let $\gamma \colon [0, + \infty) \to [0, + \infty]$ be a nondecreasing function such that $\gamma(\varepsilon) = 0$ iff
$\varepsilon = 0$, and let there exist $\gamma_0 > 0$ and $\varepsilon_0 \ge 0$ such that $\gamma(\varepsilon) \ge
\gamma_0 \beta(\varepsilon)$ for all $\varepsilon \in (0, \varepsilon_0)$.
Suppose also that the penalty function $F_{\lambda}[\phi, w, \beta]$ is globally exact, and there exists $\lambda_0$
such that the penalty function $F_{\lambda_0}[\phi, w, \gamma]$ is bounded below on $A \times \mathbb{R}_+$. Then the
penalty function $F_{\lambda}[\phi, w, \gamma]$ is globally exact.
\end{theorem}

\begin{proof}
Applying the inequality $\gamma(\varepsilon) \ge \gamma_0 \beta(\varepsilon)$, and the fact that the penalty function
$F_{\lambda}[\phi, w, \beta]$ is globally exact one gets that for any $\lambda > \lambda^*(\phi, w, \beta) / \gamma_0$
the following inequalities hold true
$$
  F_{\lambda}[\phi, w, \gamma](x, \varepsilon) \ge F_{\gamma_0 \lambda}[\phi, w, \beta](x, \varepsilon) \ge f^*
  \quad \forall (x, \varepsilon) \in A \times [0, \varepsilon_0).
$$
Then taking into account Lemma~\ref{Lemma_EpsReduction} one obtains that the penalty function 
$F_{\lambda}[\phi, w, \gamma]$ is globally exact.	 
\end{proof}

\begin{corollary}
Let there exist the right-hand side derivative of $\beta$ at $0$ such that $\beta'_+(0) > 0$. Suppose also that there
exists $\lambda_0 \ge 0$ such that the functions $F_{\lambda_0}[\phi, w, \beta]$ and $F_{\lambda_0}[\phi, w]$ are
bounded below on $A \times \mathbb{R}_+$. Then for the penalty function $F_{\lambda}[\phi, w, \beta]$ to be globally
exact it is necessary and sufficient that the penalty function $F_{\lambda}[\phi, w]$ is globally exact.
\end{corollary}

Let us illustrate Theorem~\ref{Theorem_NonlinearDepnOnEps} with a simple example.

\begin{example}
Let $\phi(t) \equiv t$, $w = 0$ and $\beta(\varepsilon) = 2 \sqrt{\varepsilon}$. Then for any $\varepsilon > 0$ 
the penalty function $F_{\lambda}[\phi, w, \beta]$ takes the form
\begin{equation} \label{Example_NonlinDepend_Eps}
  F_{\lambda}[\phi, w, \beta](x, \varepsilon) = 
  f(x) + \frac{1}{\varepsilon} d(0, \Phi(x))^2 + 2 \lambda \sqrt{\varepsilon}.
\end{equation}
From Theorems~\ref{Thrm_GlobalExactEquiv_Simple} and \ref{Theorem_NonlinearDepnOnEps} it follows that if the standard
penalty function $G_{\sigma}(x) = f(x) + \sigma d(0, \Phi(x))$ is exact, and the penalty function 
$F_{\lambda}[\phi, w, \beta]$ is bounded below on $A \times \mathbb{R}_+$ for some $\lambda \ge 0$, then the
penalty function $F_{\lambda}[\phi, w, \beta]$ is globally exact as well.

Note also that one can easily obtain a direct characterization of the exactness of the penalty function
$F_{\lambda}[\phi, w, \beta]$. Minimizing the right-hand side of (\ref{Example_NonlinDepend_Eps}) with respect to
$\varepsilon > 0$ one gets that
$$
  \min_{\varepsilon > 0} F_{\lambda}[\phi, w, \beta](x, \varepsilon) = 
  f(x) + 3 \lambda^{\frac{2}{3}} d(0, \Phi(x)^{\frac{2}{3}}.
$$
Therefore the penalty function $F_{\lambda}[\phi, w, \beta]$ is globally exact iff the penalty function 
$H_{\theta}(x) = f(x) + \theta d(0, \Phi(x))^{2/3}$ is exact, and $\lambda^*(\phi, w, \beta) = (\theta^* / 3)^{3/2}$,
where $\theta^*$ is the ex.p.p. of the penalty function $H_{\theta}$.
\end{example}

\bibliographystyle{abbrv}  
\bibliography{SmoothPenFunc_II_bibl}

\end{document}